\documentclass[a4paper,11pt,nosumlimits]{article}
\usepackage[utf8]{inputenc}
\usepackage{amsmath,amssymb,amsfonts,amsthm}
\usepackage{array} 
\usepackage{url,hyperref}
\usepackage{graphicx}
\usepackage{epstopdf}
\usepackage{color}

\renewenvironment{itemize} 
	{\begin{list}
		{$\bullet$}{\setlength{\parskip}{0cm} \setlength{\topsep}{0cm}
		 \setlength{\partopsep}{0cm} \setlength{\itemsep}{0cm} \setlength{\parsep}{0cm}\item[]}}
	{\end{list}}
 
\theoremstyle{plain}
\newtheorem{theorem}{Theorem}
\newtheorem{lemma}{Lemma}
\newtheorem{statement}{Statement}
\newtheorem{property}{Property}
\newtheorem{assumption}{Assumption}
\newtheorem{proposition}{Proposition}

\theoremstyle{definition}
\newtheorem{definition}{Definition}

\theoremstyle{remark}
\newtheorem{remark}{Remark}

\newcommand{\nneg}{[0,\infty)}
\newcommand{\pos}{(0,\infty)}
\newcommand{\ntr}{\mathbb{N}}
\newcommand{\zplus}{\mathbb{Z}_+}

\newcommand{\lf}{\lfloor}
\newcommand{\rf}{\rfloor}
\newcommand{\lc}{\lceil}
\newcommand{\rc}{\rceil}

\newcommand{\pr}{\mathbb{P}}
\newcommand{\ex}{\mathbb{E}}
\newcommand{\ind}{\mathbb{I}}

\newcommand{\eps}{\varepsilon}
\newcommand{\dlt}{\delta}
\newcommand{\lmb}{\lambda}

\newcommand{\Sgm}{\Sigma}
\newcommand{\ome}{\omega}
\newcommand{\Ome}{\Omega}

\newcommand{\sso}{{(0)}}
\newcommand{\ssn}{{(n)}}
\newcommand{\ssk}{{(k)}}
\newcommand{\ssl}{{(l)}}
\newcommand{\ssj}{{(j)}}

\newcommand{\vo}{\mathbf{0}}
\newcommand{\vone}{\mathbf{1}}
\newcommand{\ovt}{\overline{t}}
\newcommand{\ovq}{\overline{Q}}
\newcommand{\ovvq}{\overline{\mathbf{Q}}}
\newcommand{\ovb}{\overline{b}}
\newcommand{\ova}{\overline{a}}
\newcommand{\ovva}{\overline{\mathbf{a}}}

\newcommand{\cop}{\mathop{\rm =}\limits^{\text{d}}}
\newcommand{\wtil}{\widetilde}

\newcommand{\vc}{\mathbf}

\begin{document}

\title{Random Fluid Limit of an Overloaded Polling Model}
\author{ Maria Frolkova, Sergey Foss and Bert Zwart\footnote{MF~is with
CWI, P.O. Box 94079, 1098 XG Amsterdam, The Netherlands. E-mail:
\url{M.Frolkova@cwi.nl}.
 SF~is with Department of Actuarial Mathematics and Statistics, Heriot-Watt University,
 EH14 4AS Edinburgh, UK and Institute of Mathematics, Novosibirsk, Russia. E-mail: \url{
s.foss@hw.ac.uk}. BZ~is with CWI. E-mail: \url{Bert.Zwart@cwi.nl}.
BZ~is also affiliated with EURANDOM, VU University Amsterdam, and
Georgia Institute of Technology. The research of~MF and~BZ in supported by an~NWO VIDI grant.}}
\date{\today}

\maketitle

\begin{abstract}
In the present paper, we study the evolution of an~overloaded cyclic polling model that starts empty. Exploiting a~connection with multitype branching processes, we derive fluid asymptotics for the joint queue length process. Under passage to the fluid dynamics, the server switches between the queues infinitely many times in any~finite time interval causing frequent oscillatory behavior of the fluid limit in the neighborhood of zero. Moreover, the fluid limit is random. Additionally, we suggest a~method that establishes finiteness of moments of the busy period in an~$M/G/1$ queue.

\vspace{2ex} \textit{Keywords:} Cyclic polling, overload, random fluid limit, branching processes, multi-stage gated discipline, busy period moments.

\vspace{2ex} \textit{MSC2010:} Primary 60K25, 60F17; Secondary
90B15, 90B22.
\end{abstract}

\section{Introduction}

This paper is dedicated to stochastic networks called polling models. Broadly speaking, a~polling model can be defined as multiple queues served one at a~time by a~single server. As for further details --- service disciplines at the queues, routing of the server, and its walking times from one queue to another --- there exist numerous variations motivated by the wide range of applications. The earliest polling study to appear in the literature seems to be by Mack~\cite{Mack} (1957), who investigated a~problem in the British cotton industry involving a~single repairman cyclically patrolling multiple machines, inspecting them for malfunctioning and repairing them. Over the past few decades, polling techniques have been of extensive use in the areas of computer and communication networks as well as manufacturing and maintenance. Along with that, a~vast body of related literature has grown. For overviews of the available results on polling models and their analysis methodologies, we refer the reader to Takagi~\cite{Takagi86, Takagi90, Takagi97}, Boxma~\cite{Boxma91}, Yechiali~\cite{Yechiali} and Borst~\cite{Borst96}.

Across the great variety of polling models, there exists the ``classical" one, which was first used in the analysis of time-sharing computer systems in the early 70's. This model is {\it cyclic}, i.e. if there are $I$ queues in total, they are visited by the server in the cyclic order $1,2,\ldots,I,1,2,\ldots$. All of the queues are supposed to be infinite-buffer queues, and to each of them there is a~Poisson stream of customers with i.i.d.\ service times. After all visits to a~queue, i.i.d.\ walking, or switchover, times are incurred. All interarrival times, service times and switchover times are mutually independent, and their distributions may vary from queue to queue as well as the service disciplines. Examples of the most common service disciplines are {\it exhaustive} (the queue is served until it becomes empty), {\it gated} (in the course of a~visit, only those customers get served who are present in the queue when the server arrives to, or~{\it poll}, the queue), and {\it $k$-limited} (at most $k$ customers get served per visit). The present paper is also centered around the classical polling model. We assume zero switchover times and allow a~wide class of service disciplines that includes both exhaustive and gated policies and is discussed later in more detail.

Amongst~desirable properties of any service system, the first one is stability. So, naturally, the major part of the polling related literature is focused on the performance of stable models. Foss \& Kovalevskii~\cite{FK99} obtained~an interesting result of null recurrence over a~thick region of parameter space for a~two-server modification of polling. MacPhee {\it et~al.} \cite{MacPhee07,MacPhee08} have recently observed the same phenomenon for a~hybrid polling/Jackson network, where the service rate and customer rerouting probabilities are randomly updated each time the server switches from one queue to another.

The study of critically loaded polling models was initiated about two decades ago by Coffman {\it et~al.} \cite{Coffman95, Coffman98}, who proved a~so called averaging principle: in the diffusion heavy traffic limit, certain functionals of the joint workload process can be expressed via the limit total workload, which was shown to be a~reflected Brownian motion and a~Bessel process in the case of zero and non-zero switchover times, respectively.  In subsequent years, the work has been carried on by Kroese~\cite{Kroese}, Vatutin \& Dyakonova~\cite{Vatutin}, Altman \& Kushner~\cite{Altman}, van der Mei~\cite{Mei07} and others. In particular, heavy-traffic approximations of the steady state and waiting time distributions have been derived.

Although overloaded service systems are an~existing reality and it is of importance to control or predict how fast they blow up over time, to the best of our knowledge, for polling models this problem has not been addressed in the literature so far. The present paper aims to fill in the gap. Moreover, this appears to be a~really exciting problem because it reveals the following unusual phenomenon. Our interest is in fluid approximations, i.e. the limit of the scaled joint queue length process
\[
(Q_1, \ldots, Q_I)(x^\ssn \cdot) / x^\ssn
\]
along a~deterministic sequence $x^\ssn \to \infty$. Remarkably, in contrast to the many basic queueing systems with deterministic fluid limits, overloaded polling models preserve some randomness under passage to the fluid dynamics. Other examples of simplistic designs combined with random fluid limits are two-queue two-server models of Foss \& Kovalevskii~\cite{FK99} and Kovalevskii {\it et al.}~\cite{KTF05}. To the latter work~\cite{KTF05}, we refer for an~insightful discussion of  the nature of randomness in fluid limits in general and for an~overview of the publications on the topic.

To illustrate the key idea that has led us to the result, consider the simple, symmetric model of~$I=2$ queues with exhaustive service, zero switchover times and empty initial condition (without the last assumption, the analysis becomes much simpler). In isolation, the queues are stable, and the whole system is overloaded, i.e. $1/2 < \lmb / \mu < 1$, where $\lmb$ and $1/\mu$ are the arrival rate and the mean service time, respectively (in both queues). Denote the supposedly existing limit queue length process by $(\ovq_1, \ovq_2)(\cdot)$. Note that, given the limit size of the queue in service at any non-zero time instant, the entire trajectories of both queues can be restored by the SLLN. Indeed, the limit total population $(\ovq_1 + \ovq_2)(\cdot)$ grows at rate $2 \lmb - \mu$. Because of the symmetry, at any fixed time $T > 0$, the queues might (in the limit) be in service 
with equal probabilities, let it be queue~$1$. Then in Fig.~\ref{fig:intro} the limit queues~$1$ and $2$ follow the solid and dashed trajectories, respectively. Starting from time~$T$, the limit queue~$1$ gets cleared up at rate~$\lmb - \mu$ until it becomes empty, say, at time~$\ovt^{(1)}$. Since~$\ovt^{(1)}$, when the limit total population $(2 \lmb - \mu)\ovt^{(1)}$ comes from queue~$2$ alone, queue~$2$  gets cleared up at rate~$\lmb - \mu$ until it becomes empty at time~$\ovt^{(2)}$, while queue~$1$ grows at the arrival rate~$\lmb$. Moving forward and backward in this way, one can continue the two trajectories onto~$[T, \infty)$ and $(0,T]$, respectively, and see that they oscillate at an~infinite rate when approaching zero. Now, the same algorithm applies if $t^{(1)}$ is known, which is the first switching instant after~$T$, and the following crucial observation makes it possible to find the distribution of~$t^{(1)}$ (so the randomness of $t^{(1)}$ makes the fluid limit random). Let customer~2 to be a~descendant of customer~1 if customer~2 arrives to the system while customer~1 is receiving service, or customer~2 is a~descendant of a~descendant of customer~$1$. Then the size of the non-empty queue at switching instant form a branching process.

\begin{figure}[!htb] 
\centering
\includegraphics[scale=0.7]{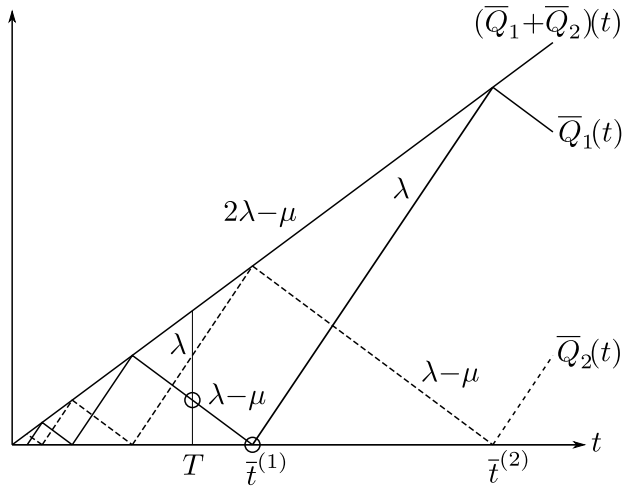}
\caption{Fluid limit of a symmetric two-queue model with exhaustive service}
\label{fig:intro}
\end{figure}

The idea of representing arriving customers as descendants of the customer in service, has appeared in Foss~\cite{branchingFoss} in the studies of an~extension of Klimov's $\mu c$-rule, and then in Resing~\cite{Resing93}, who introduced a~wide class of service disciplines that, for the classical polling model (more general periodic server routing is also allowed), guarantee the joint queue length at the successive polling instants of a~fixed queue to form a multitype branching process (MTBP). This embedded MTBP is the cornerstone of the analysis that we carry out in this paper.

We now describe the class of service disciplines that we allow in this paper. It is a~subclass of the MTBP-policies, and we call them {\it multigated} meaning that each visit of each queue consists of a~number of consecutive gated service phases. The upper bound on the number of phases, called the {\it gating index}, comes from the input data (together with the interarrival and service times). Gating indices for different visits of the same queue are i.i.d.\ random variables whose distribution may vary from queue to queue, and gating indices for different queues are mutually independent. Gating indices equal $1$ and $\infty$ correspond to exhaustive and the classical gated service, respectively. Multigated policies with deterministic gating indices were studied (and, in~fact, introduced) recently by van~Wijk {\it et al.}~\cite{Sandra12} with the purpose of balancing fairness and efficiency of polling models. Van der Mei and co-authors~\cite{vdMRe07, vdMRo11} consider multi-stage gated policies, but those are different than in~\cite{Sandra12} and here.

Throughout the paper, we consider the case of zero switchover times. The case of non-zero switchover times can be treated with similar methods.

As for the proofs, multiple asymmetric queues with non-exhaustive service create more work compared to the simple two-queue example discussed above. Knowing the limit total population is of little use now since it only reduces the dimension of the problem by one. We show that, in general situation, the fluid limit queue length trajectory $(\ovq_1, \ldots, \ovq_I)(\cdot)$ is determined by~$2I$ random parameters: the earliest polling instants $\ovt_1, \ldots, \ovt_I$ that, in the limit, follow a~fixed time instant, and the limit sizes $\ovq_1(\ovt_1), \ldots, \ovq_I(\ovt_I)$ of the corresponding polled queues. The overload assumption and multigated policies provide the framework of supercritical MTBP's, and we can apply the Kesten-Stigum theorem~\cite{KestenStigum66, KestenStigum96} (the classical result on asymptotics of supercritical MTBP's) to find the distribution of, for example, $(\ovq_1, \ldots, \ovq_I)(\ovt_1)$. Then suitable SLLN's imply that the the other parameters $\ovt_1, \ldots, \ovt_I$, $\ovq_2(\ovt_2), \ldots, \ovq_I(\ovt_I)$ can be expressed either via the Kesten-Stigum limit $(\ovq_1, \ldots, \ovq_I)(\ovt_1)$ or via each other. Note also that the Kesten-Stigum theorem requires certain moments of the offspring distribution to be finite. The visit at a~queue is the longest when service is exhaustive, implying more customers in the other queues in the end of the cycle. So attempts to satisfy the moment conditions of the Kesten-Stigum theorem boil down to proving finiteness of the corresponding moment for the busy period of an~$M/G/1$ queue, which is an~interesting and novel result by itself. Besides, we obtain an~estimate for this moment, and our approach is valid for a~wide class of regularly varying convex functions, in particular power and logarithmic functions.

The rest of the paper is organised as follows. Section~\ref{sec:model_description} describes the cyclic polling model and the class of service disciplines. Section~\ref{sec:MTBP} explains the connection between the model and MTBP's, gives some preliminaries from the theory of MTBP's and derives characteristics of the embedded MTBP. In Section~\ref{sec:fluid_limit}, we state our main result --- the fluid limit theorem --- and discuss the optimal representation of the fluid limit from the computational point of view (Remark~\ref{rem:optimal_computation}). Section~\ref{sec:proof_MTBP} proves the results of Section~\ref{sec:MTBP}, see the proof of Lemma~\ref{lem:log} for estimates on the moments of the busy period of an~$M/G/1$ queue. Section~\ref{sec:proof_fluid_limit} proves the fluid limit theorem. Proofs of some auxiliary statements are given in the Appendix.

\paragraph{Notation} 

With $x:=y$ we mean that $x$ is defined as equal to $y$. 

The standard sets are: positive integers $\ntr := \{1,2,\ldots\}$, non-negative integers $\zplus := \ntr \cup \{0\}$, integers $\mathbb{Z} = \{ 0, \pm1, \pm2, \ldots \}$.

All vectors are $I$-dimensional row vectors, $\cdot^T$ denotes the operation of transposition. All vector notations are boldface. The vector with all coordinates equal~$0$ is denoted by~$\vo$, with all coordinates equal~$1$ by~$\vone$, and with coordinate~$i$ equal~$1$ and the other coordinates equal~$0$ by~$\vc{e}_i$. The following operations are defined on vectors $\vc{x} = (x_1, \ldots, x_I)$, $\vc{y} = (y_1, \ldots, y_I)$,
\begin{itemize}
\item partial order: $\vc{x} \leq \vc{y}$ if $x_i \leq y_i$ for all $i$;
\item $L_1$-norm $|\vc{x}| = \sum_{i=1}^I x_i$;
\item coordinate-wise product  $\vc{x} \times \vc{y} = (x_1 y_1, \ldots, x_I y_I)$;
\item power: if all $x_i > 0$, then  $\vc{x}^\vc{y} = \prod_{i=1^I} x_i^{y_i}$;
\item binomial coefficient: if $\vc{x}, \vc{y} \in \zplus^I$ and $\vc{y} \leq \vc{x}$, $\dbinom{\vc{x}}{\vc{y}} = \prod_{i=1}^I \dbinom{x_i}{y_i}  = \prod_{i=1}^I \dfrac{x_i!}{y_i!(x_i - y_i)!}$.
\end{itemize}

For a~real $x$, let $\lf x \rf $ be its maximum integer lower bound, $\lc x \rc $ its minimum integer upper bound, and put $\{ x \}= x - \lf x \rf$.

If a~superscript is in parentheses, then it is an upper index, otherwise a~power.

If random objects $X$ and $Y$ are equal in distribution, we write $X \cop Y$ and say that $X$ is a~copy of~$Y$.

\section{Model description} \label{sec:model_description}
This section contains a detailed description of the cyclic polling model and the class of service disciplines that we allow for this model. It also specifies the stochastic assumptions. All stochastic primitives introduced throughout the paper are defined on a common probability space $(\Omega, \mathcal{F}, \pr)$ with expectation operator $\ex$. 

\paragraph{Cyclic polling} Consider a~system that consists of multiple infinite-buffer queues labeled by $i = 1, \dots, I$, where $I$ is finite, and a~single server. There are external arrivals of customers to the queues that line up in the corresponding buffers in the order of arrival. The server idles if and only if the entire system is empty. While the system is non-empty, the server works at unit rate serving one queue at a~time and switching from one queue to another in the cyclic order: after a~period of serving queue~$i$, called a {\it visit to queue~$i$}, a~visit to queue~$i \!\! \mod I+1$ follows. Note that, while the system is non-empty, empty queues get visited as well in the sense that, once the server arrives to (or, {\it polls}) an~empty queue, say at time~$t$, it has to leave immediately, and the visit in this case is defined to be the empty interval $[t,t)$. Now suppose that, at a~particular time instant, the system empties upon completion of a~non-empty visit to queue~$i$. For mathematical convenience, we assume that such an~instant is followed by a~single (empty) visit to each of the empty queues~$i+1, \ldots, I$. Then the server idles until the first arrival into the empty system. If that arrival is to queue~$i$, a~single (empty) visit to each of the empty queues~$1, \ldots, i-1$ precedes the visit to queue~$i$. In the course of a~visit, a~number of customers at the head of the queue get served in the order of arrival and depart. The service disciplines at the queues specify how many customers should get served per visit, we now proceed with their description.

\paragraph{Multigated service} With multigated service in a~queue we mean that each visit of that queue consisits of a number of consecutive gated service phases. More formally, we say that {\it the server gates a queue} at a~particular time instant meaning that the queue is in service at the moment, and all the customers found in the queue at the moment are guaranteed to receive service during the current visit. Customers gated together are served in the order of arrival. For each visit, its {\it gating index} is defined: it is the upper bound on the number of times the server is supposed to gate the queue in the course of the visit. The gating indices for different queues and for different visits of the same queue might be different. The first time during a~visit when the server gates the queue is upon polling the queue. The other gating instants are defined by induction: as soon as the customers found in the queue the last time it was gated have been served, the queue is gated again provided that the total number of gating procedures is not going to exceed the gating index. If the queue is empty upon gating, the server switches to the next queue, and thus the actual number of gating procedures performed during a~visit might differ from the gating index for that visit. Now we define a~generic multigated service discipline.

\begin{definition}
Let a random variable $X$ take values in~$\zplus \cup \{ \infty \}$. The service discipline at a~particular queue is called {\it $X$-gated} if the gating indices for different visits of this queue are i.i.d.\ copies of $X$. If a~gating index equals $0$,  the server should leave immediately after polling the queue. The values $1$ and $\infty$ of a~gating index correspond to conventional gated and exhaustive service, respectively.
\end{definition}

\begin{remark}
Multigated service disciplines guarantee the population of the polling system at polling instants of a fixed queue to an~MTBP, laying the foundation for the analysis that we carry out in this paper. We discuss this connection with MTBP's in~detail in the next section.
\end{remark}

\paragraph{Stochastic assumptions} We consider the cyclic polling system described above to evolve in the continuous time horizon $t \in \nneg$. At $t = 0$, the system is empty. Arrivals of customers to queue~$i$ form a~Poisson process $E_i(\cdot)$ of rate $\lmb_i$. Introduce also the vector of arrival rates
\[
\boldsymbol{\lmb} := (\lmb_1, \ldots, \lmb_I).
\]
Service times of queue~$i$ customers are drawn from a~sequence $\{ B_i^\ssn \}_{n \in \ntr}$ of i.i.d.\\ copies of a~positive random variable $B_i$ with a finite mean value $1/\mu_i$. Gating indices for queue~$i$ are drawn from a sequence $\{ X_i^\ssn \}_{n \in \ntr}$ of i.i.d.\ copies of a~random variable $X_i$ taking values in~$\zplus \cup \{\infty\}$. The random elements $E_i(\cdot)$, $\{ B_i^{\ssn} \}_{n \in \ntr}$ and $\{ X_i^\ssn \}_{n \in \ntr}$, $i = 1, \ldots, I$, are mutually independent. Additionally, we impose the following conditions on the load intensities and service times.

\begin{assumption} \label{ass:load}
For all~$i$, $\lmb_i / \mu_i < 1$, and $\sum_{i=1}^I \lmb_i / \mu_i > 1$.
\end{assumption}

\begin{assumption} \label{ass:log}
For all~$i$, $\ex B_i \log B_i < \infty$.
\end{assumption}

We study the system behavior in terms of its queue length process
\[
\vc{Q}(\cdot) = (Q_1, \ldots, Q_I)(\cdot),
\]
where $Q_i(t)$ is the number of customers in queue~$i$ at time~$t$.

\section{Connection with MTBP's} \label{sec:MTBP}
This section is devoted to a multitype branching process (MTBP) embedded in the queue length process $\vc{Q}(\cdot)$ and enabling its further analysis.

To start with, we divide the time horizon into pairwise-disjoint finite intervals in such a~way that each interval includes a~single (possibly, empty) visit of the server to each of the queues starting from the first one. Let
\begin{align*}
\nneg &= \bigcup_{n \in \zplus} [t^\ssn, t^{(n+1)}), \\
[t^\ssn, t^{(n+1)}) &= [t^\ssn, t^\ssn_1) \bigcup_{i=1}^I [t^\ssn_i,t^\ssn_{i+1}),
\end{align*}
where 
\begin{itemize}
\item $t^{(0)}=0$ and  $t^\ssn \leq t_1^\ssn \leq \ldots \leq t_{I+1}^\ssn = t^{(n+1)}$;
\item if the system is empty at $t^\ssn$, then the interval $[t^\ssn, t^\ssn_1)$ is the period of waiting until the first arrival, otherwise $t^\ssn = t^\ssn_1$;
\item the interval $[t^\ssn_i,t^\ssn_{i+1})$ is the visit to queue~$i$ following~$t^\ssn$, with $t^\ssn_i = t^\ssn_{i+1}$ if the visit is empty.
\end{itemize}

The interval $[t^\ssn, t^{(n+1)})$ is called {\it session~$n$}. The interval $[t^\ssn_i,t^\ssn_{i+1})$ is called {\it visit~$n$ to queue~$i$}, and the gating index for this visit is $X^\ssn_i$.

For multigated service disciplines that we consider in this paper, the following holds.
\begin{property} \label{pty:branching}
For all $i = 1, \ldots, I$, the customers found in queue~$i$ at a~polling instant get replaced during the course of the visit by i.i.d.\ copies of a~random vector $\check{\vc{L}}_i  = (\check{L}_{i,1},\ldots,\check{L}_{i,I})$ that has the distribution of $\vc{Q}(t_{i+1}^\ssn)$ given that $\vc{Q}(t_i^\ssn) = \vc{e}_i$ \textup{(}this distribution does not depend on~$n$\textup{)}.
\end{property}

By Resing~\cite{Resing93}, Property~\ref{pty:branching} implies that the sequence 
\[
\{\vc{Q}(t^\ssn)\}_{n \in \zplus}
\]
forms an~MTBP with immigration in state~$\vo$. In the rest of the section, we introduce a number of objects associated with this MTBP and discuss some of its properties.

The random vector $\check{L}_i$ mentioned in Property~\ref{pty:branching} we call the {\it visit offspring of a~queue~$i$ customer}. Define also the {\it visit duration at~queue~$i$} to be a~random variable $V_i $ equal in distribution to $t_{i+1}^\ssn - t_i^\ssn$ given that $Q_i(t_i^\ssn) = 1$, and the {\it session offspring of a~queue~$i$ customer} to be a~random vector $\vc{L}_i = (L_{i,1}, \ldots, L_{i,I})$ that has the distribution of $\vc{Q}(t^{(n+1)})$ given that $\vc{Q}(t^\ssn) = \vc{e}_i$. Then the immigration distribution is given by
\[
G(\vc{k}) := \pr \{ \vc{Q}(t^{(n+1)}) = \vc{k} | \vc{Q}(t^\ssn) = \vc{0} \} = \sum_{i=1}^I \lmb_i \pr \{ \vc{L}_i = \vc{k} \} / \sum_{i=1}^I \lmb_i, \quad \vc{k} \in \zplus^I.
\]

The following lemma computes the mean values
\[
\gamma_i := \ex V_i, \quad 
\check{\vc{m}}_i = (\check{m}_{i,1}, \ldots, \check{m}_{i,I}) := \ex \check{\vc{L}}_{i}, \quad 
\vc{m}_i = (m_{i,1}, \ldots, m_{i,I}) := \ex \vc{L}_i.
\]

\begin{lemma} \label{lem:mean}
For all~$i$, 
\[
\check{m}_{i,i} = \ex (\lmb_i / \mu_i)^{X_i} \quad \text{and} \quad \gamma_i =  \dfrac{1 - \check{m}_{i,i}}{\mu_i - \lmb_i},
\]
and, for $i \neq j$,
\[
\check{m}_{i,j} = \lmb_j \gamma_i.
\]
For the $m_{i,j}$'s, there is a recursive formula:
\[
m_{I,j} = \check{m}_{I,j} \quad \text{for all~$j$}, 
\]
and, for $i = 1, \ldots, I-1$, $\vc{m}_i$ is computed via $\vc{m}_{i+1}$,
\[
m_{i,j} = \check{m}_{i,j} \ind \{ i \geq j \} + \sum_{k=i+1}^I \check{m}_{i,k} m_{k,j} \quad \text{for all~$j$}.
\]
\end{lemma}

The proof follows in Section~\ref{sec:lem_mean}.

By the Perron-Frobenius theorem (see e.g. \cite[Theorem~5.1]{Harris}), the {\it mean session offspring matrix $M := \{ m_{i,j} \}_{i,j=1}^I$} has a~positive eigenvalue $\rho$ that is greater in absolute value than any other eigenvalue of $M$. The eigenspace associated to $\rho$ is one-dimensional and parallel to a~vector with all coordinates positive. Then there exist vectors $\vc{u} = (u_1, \ldots, u_I)$ and $\vc{v} = (v_1, \ldots, v_I)$ with all coordinates positive such that 
\[
M \vc{u}^T = \rho \vc{u}^T, \quad \vc{v} M = \rho \vc{v} \quad \text{and} \quad \vc{v} \vc{u}^T = 1.
\]

Now introduce an~auxiliary MTBP $\{ \vc{Z}^\ssn \}_{n \in \zplus}$ with no immigration and such that, given $\vc{Z}^\ssn = \vc{e}_i$, the next generation $\vc{Z}^{(n+1)}$ is equal in distribution to~$\vc{L}_i$. Denote by $q_i$ the {\it extinction probability} for the process $\{ \vc{Z}^\ssn \}_{n \in \zplus}$ given that $\vc{Z}^{(0)} = \vc{e}_i$, and introduce the vector of extinction probabilities
\[
\vc{q} := (q_1, \ldots, q_I).
\]

Then the probability for the process $\{\vc{Q}(t^{(n)})\}_{n \in \zplus}$ to return to~$\vo$ is given by
\[
q_G := \sum_{\vc{k} \in \zplus^I} G(\vc{k}) \vc{q}^\vc{k}.
\]

\begin{remark}
Since all time instants $t$ such that $\vc{Q}(t) = \vo$ are contained among the $t^\ssn$'s, the probability for the process $\vc{Q}(\cdot)$ to return to~$\vo$ equals $q_G$, too.
\end{remark}

By Assumption~\ref{ass:load}, the MTBP's $\{ \vc{Q}(t^\ssn) \}_{n \in \zplus}$ and $\{ \vc{Z}^\ssn \}_{n \in \zplus}$ are supercritical (the proof is postponed to the Appendix).
\begin{lemma} \label{lem:supercritical}
For the Perron-Frobenius eigenvalue $\rho$ and the extinction probabilities~$q_i$, we have $\rho > 1$ and $q_i < 1$ for all~$i$. By the latter, $q_G <1$, too.
\end{lemma}

Assumption~\ref{ass:log} guarantees finiteness of the corresponding moments for the offspring distribution of the MTBP's $\{ \vc{Q}(t^\ssn) \}_{n \in \zplus}$ and $\{ \vc{Z}^\ssn \}_{n \in \zplus}$ (see Section~\ref{sec:lem_log} for the proof).

\begin{lemma} \label{lem:log}
For all~$i$ and~$j$, $\ex L_{i,j} \log L_{i,j} < \infty$, where $0 \log 0 := 0$ by convention.
\end{lemma}

Finally, we quote the Kesten-Stigum theorem for supercritical MTBP's (see e.g. \cite{KestenStigum66, KestenStigum96}), which is our starting point when proving the convergence results of the next section. By that theorem and Lemmas~\ref{lem:supercritical} and~\ref{lem:log}, the auxiliary process~$\{ \vc{Z}^\ssn \}_{n \in \zplus}$ has the following asymptotics.
\begin{proposition} \label{th:Kesten_Stigum}
Given $\vc{Z}^{(0)} = \vc{e}_i$,
\[
\vc{Z}^\ssn / \rho^n \to \zeta_i \vc{v} \quad \text{a.s.} \quad \text{as $n \to \infty$}, 
\]
where the distribution of the random variable~$\zeta_i$ has a~jump of magnitude $q_i<1$ at~$0$ and a~continuous density function on~$\pos$, and $\ex \zeta_i = u_i$. 
\end{proposition}

\section{Fluid limit} \label{sec:fluid_limit}

In this section, we present out main result which concerns the behavior of the model under study on a~large time scale. 

For each~$n \in \zplus$, introduce the scaled queue length process
\begin{equation} \label{eq:ovvq}
\ovvq^\ssn(t) := \vc{Q}(\rho^n t) / \rho^n, \quad t \in \nneg.
\end{equation}
We are interested in the a.s. limit of the processes~\eqref{eq:ovvq} as $n \to \infty$, which we call the~{\it fluid limit} of the model. It appears that, in order to precisely describe the fluid limit, the information provided by the following theorem is sufficient.

For $n \in \mathbb{Z}$, let
\[
\eta_n :=
\left\{
\begin{array}{ll}
 \min \{ k \colon t^\ssk \geq \rho^n \} & \text{if $n \geq 0$}, \\
0 & \text{if $n < 0$}.
\end{array}
\right.
\]
\begin{theorem} \label{th:eta_n}
There exist constants  $\ovb_i \in \pos$ and $\ovva_i = (\ova_{i,1}, \ldots, \ova_{i,I}) \in \nneg^I$, $i = 1, \ldots, I+1$, and a~random variable~$\xi$ with values in $[1,\rho)$ such that, for all $k \in \zplus$ and~$i$, \begin{equation} \label{eq16}
t_i^{(\eta_n + k)} / \rho^n \to  \rho^k \ovb_i \xi \quad \text{and} \quad 
\vc{Q}(t_i^{(\eta_n + k)}) / \rho^n \to \xi \rho^k \ovva_i  \quad \text{a.s. as $n \to \infty$}.
\end{equation}
The $\ovb_i$'s and $\ovva_i$'s are given by
\begin{equation} \label{eq:ovb_i}
\ovb_1 = 1, \quad \ovb_{i+1} = \ovt_i + (v_i / \alpha + \lmb_i(\ovb_i - \ovb_1)) \gamma_i, \quad i = 1, \ldots, I,
\end{equation}
and
\begin{equation} \label{eq:ova_i}
\ovva_1 = v/\alpha \quad \ovva_{i+1} = \ovva_i + (\ovb_{i+1} - \ovb_i) \boldsymbol{\lmb} - (\ovb_{i+1} - \ovb_i) \mu_i \vc{e}_i, \quad i = 1, \ldots, I,
\end{equation}
where
\[
\alpha = \frac{\sum_{i=1}^I v_i / \mu_i}{\sum_{i=1}^I \lmb_i / \mu_i - 1}.
\]

The distribution of $\xi$ is given by
\begin{align*}
\pr \{ \xi \geq x \} = \frac{1}{1 - q_G} \sum_{\begin{subarray}{l} \vc{k} \in \zplus^I, \\ |\vc{k}| \geq 1 \end{subarray}} G(\vc{k}) 
\sum_{ \begin{subarray}{l} \vc{l} \leq \vc{k}, \\ |\vc{l}| \geq 1 \end{subarray}} \binom{\vc{k}}{\vc{l}}(\vone - \vc{q})^\vc{l} \vc{q}^{\vc{k}-\vc{l}} \times& \\
\times \pr \{ \{ \log_\rho (\alpha \sum_{i=1}^I \sum_{j=1}^{l_i} \xi_i^{(j)}) \}  \geq  \log_\rho x\}&, \quad x \in [1,\rho),
\end{align*}
where $\xi_i^\ssj$, $j \in \ntr$, are i.i.d.\ random variables with the distribution of $\zeta_i$ given that $\zeta_i > 0$, and the sequences $\{ \xi_i^\ssj \}_{j \in \ntr}$, $i = 1, \ldots, I$, are mutually independent.
\end{theorem}

The proof of Theorem~\ref{th:eta_n} combines the Kesten-Stigum theorem with various dynamic equations and laws of large numbers, see Section~\ref{sec:proof_fluid_limit}.

\begin{remark}
Since $t_{I+1}^\ssn = t_1^\ssn$, we also have
\[
\ovb_{I+1} = \rho \ovb_1 \quad \text{and} \quad \ovva_{I+1} = \rho \ovva_1.
\]
\end{remark}

\begin{remark}
There is an~alternative way to compute the $\ova_i$'s:
\[
\ovva_1 = v, \quad \ovva_{i+1} = \ovva_i - \ovva_{i,i} \vc{e}_i + \ova_{i,i} \check{\vc{m}}_i, \quad i = 1, \ldots, I,
\]
which implies that $\ova_{i,j} > 0$ if $|i - j| \neq 1$ and $\ova_{i,i+1} = 0$ if and only if the service discipline at~queue~$i$ is exhaustive. See Lemma~\ref{lem:t_i^n} and Remark~\ref{rem:a_ova} in Section~\ref{sec:preliminary}.
\end{remark}

Based on the results of Theorem~\ref{th:eta_n}, Theorem~\ref{th:fluid_limit} below derives the fluid limit equations from the suitable dynamic equations, see Section~\ref{sec:proof_fluid_limit} for the proof.

\begin{theorem} \label{th:fluid_limit}
There exists a~deterministic function $\ovvq(\cdot) = (\ovq_1, \ldots, \ovq_I)(\cdot) \colon \nneg \to \nneg^I$ such that,
\[
\text{a.s. as $n \to \infty$}, \quad \ovvq^\ssn(\cdot) \to \xi \ovvq (\cdot / \xi) \quad \text{uniformly on compact sets},
\]
where the random variable $\xi$ is defined in Theorem~\ref{th:eta_n}.

The function $\ovvq(\cdot)$ is continuous and piecewise linear and given by
\begin{equation} \label{eq:ovQ_v1}
\ovvq(t)=\left\{
\begin{array}{ll}
0 & \text{if $t = 0$} \\
\rho^k \ovva_i + (t - \rho^k \ovb_i) \boldsymbol{\lmb} - (t - \rho^k \ovb_i) \mu_i \vc{e}_i & \text{if $t \in [\rho^k \ovb_i, \rho^k \ovb_{I+1})$, $i = 1, \ldots, I$, $k \in \mathbb{Z}$},
\end{array}
\right.
\end{equation}
or, equivalently, by
\begin{equation} \label{eq:ovQ_v2}
\ovq_i(t)=\left\{
\begin{array}{ll}
0 & \text{if $t = 0$}, \\
\rho^k \ova_{i,i} + (\lmb_i - \mu_i)(t - \rho^k \ovb_i)& \text{if $t \in [\rho^k \ovb_i, \rho^k \ovb_{i+1})$, $k \in \mathbb{Z}$}, \\
\rho^{k+1} \ova_{i,i} - \lmb_i (\rho^{k+1} \ovb_i - t)& \text{if $t \in [\rho^k \ovb_{i+1}, \rho^{k+1} \ovb_i)$, $k \in \mathbb{Z}$},
\end{array}
\right.
\quad i = 1, \ldots,I.
\end{equation}
\end{theorem}

\begin{remark} \label{rem:optimal_computation}
By~\eqref{eq:ovQ_v2}, the whole process $\ovvq(\cdot)$ is defined by the constants $\ovb_i$ and $\ova_{i,i}$. The fastest way to compute the $\ovb_i$'s and $\ova_{i,i}$'s is using the simultaneous recursion
\[
\ovb_1 = 1, \quad \ova_{i,i} = v_i/\alpha + \lmb_i(\ovb_i - \ovb_1), \quad \ovb_{i+1} = \ovb_i + \ova_{i,i} \gamma_i, \quad i = 1, \ldots, I.
\]
See the last part of the proof of Lemma~\ref{lem:t_i^n} (namely, \eqref{eq13} and \eqref{eq14}) and Remark~\ref{rem:a_ova} in Section~\ref{sec:preliminary}.
\end{remark}

Finally, Fig.~\ref{fig:fluid_limit} depicts a trajectory of the limiting process $\xi \ovq(\cdot / \xi)$.
\begin{figure}[!htb] 
\centering
\includegraphics[scale=0.7]{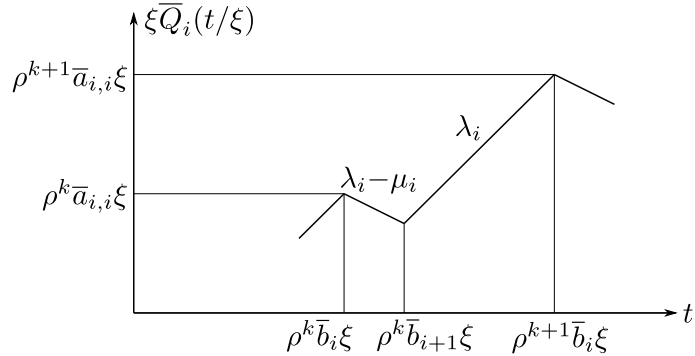}
\caption{Fluid limit of queue~$i$}
\label{fig:fluid_limit}
\end{figure}

\section{Proofs for Section~\ref{sec:MTBP}} \label{sec:proof_MTBP}

Here we prove the properties of the offspring distribution of the embedded MTBP~$\{\vc{Q}(t^\ssn)\}_{n \in \zplus}$.

\subsection{Proof of Lemma~\ref{lem:mean}} \label{sec:lem_mean}
First we compute the $\gamma_i$'s. For $k \in \zplus \cup \{ \infty \}$, let a~random variable $V_i(k)$ be the visit duration at queue~$i$ given that the service discipline at queue~$i$ is $k$-gated. 

Recall that the gating index equal $\infty$ corresponds to exhaustive service, and hence
\[
\ex V_i(\infty) = 1 / (\mu_i - \lmb_i).
\] 
Now note that 
\begin{equation} \label{eq20}
V_i(0) = 0 \quad \text{and}  \quad V_i(k+1) \cop B_i + \sum_{l=1}^{E_i(B_i)} V_i^\ssl(k), \quad k \in \zplus.
\end{equation}
where the random elements $B_i$, $E_i(\cdot)$ and $\{ V_i^\ssl(k) \}_{l \in \ntr}$ are mutually independent, and $V_i^\ssl(k)$, $l \in \ntr$, are i.i.d.\ copies of $V_i(k)$. Then, for $k \in \zplus$,
\begin{equation} \label{eq22}
\begin{split}
\ex V_i(k+1) &= \frac{1}{\mu_i}+\frac{\lmb_i}{\mu_i} \ex V_i(k) 
= \frac{1}{\mu_i} \left( 1+\frac{\lmb_i}{\mu_i} \right) + \left( \frac{\lmb_i}{\mu_i} \right)^2 \ex V_i(k-1) = \ldots \\
&= \frac{1}{\mu_i} \left( 1+\frac{\lmb_i}{\mu_i} + \ldots + \left( \frac{\lmb_i}{\mu_i} \right)^k \right) + \left( \frac{\lmb_i}{\mu_i} \right)^{k+1} \ex V_i(0)
= \frac{1}{\mu_i} \frac{1 - (\lmb_i/\mu_i)^{k+1}}{1 - \lmb_i/\mu_i},
\end{split}
\end{equation}
and
\[
\gamma_i = \sum_{k \in \zplus \cup \{ \infty \}} \pr\{ X_i = k \} \ex V_i(k) = \frac{1 - \ex (\lmb_i/\mu_i)^{X_i}}{\mu_i - \lmb_i}.
\]

In a~similar way, we compute the $\check{m}_{i,i}$'s. For $k \in \zplus \cup \{ \infty \}$, let a~random variable $\check{L}_{i,i}(k)$ be the queue~$i$ visit offspring of a~queue~$i$ customer given that the service discipline at queue~$i$ is $k$-gated. Since 
\[
\check{L}_{i,i}(\infty) = 0, \quad \check{L}_{i,i}(0) = 1 \quad \text{and} \quad \check{L}_{i,i}(k+1) \cop \sum_{l=1}^{E_i(B_i)} \check{L}_{i,i}^\ssl(k), \quad k \in \zplus,
\]
where the random elements $B_i$, $E_i(\cdot)$ and $\{ \check{L}^\ssl_{i,i}(k) \}_{l \in \ntr}$ are mutually independent, and $\check{L}^\ssl_{i,i}(k)$, $l \in \ntr$, are i.i.d.\ copies of $\check{L}_{i,i}(k)$, we have
\[
\ex \check{L}_{i,i}(k+1) = (\lmb_i / \mu_i) \ex \check{L}_{i,i}(k) = \ldots = (\lmb_i / \mu_i)^{k+1}, \quad  k \in \zplus,  \quad \text{and} \quad \check{m}_{i,i} = \ex (\lmb_i/\mu_i)^{X_i}.
\]

The formulas for the $\check{m}_{i,j}$'s, $i \neq j$, and the $m_{i,j}$'s follow, respectively, by the representations
\begin{equation} \label{eq23}
\check{L}_{i,j} \cop E_j(V_i), \quad i \neq j,
\end{equation}
where $V_i$ and $E_j(\cdot)$ are independent, and
\begin{equation} \label{eq24}
L_{i,j} \cop \left\{
\begin{array}{rl}
\check{L}_{i,j} + \sum_{l=1}^{\check{L}_{i,i+1}} L_{i+1,j}^\ssl + \ldots + \sum_{l=1}^{\check{L}_{i,I}} L_{I,j}^\ssl, & i \geq j,\\
\sum_{l=1}^{\check{L}_{i,i+1}} L_{i+1,j}^\ssl + \ldots + \sum_{l=1}^{\check{L}_{i,I}} L_{I,j}^\ssl, &i < j,
\end{array}
\right.
\end{equation}
where $L_{i,j}^\ssl$, $l \in \ntr$, are i.i.d.\ copies of $L_{i,j}$, and the sequences $\{ L_{i,j}^\ssl \}_{l \in \ntr}$, $i,j = 1, \ldots, I$, are mutually independent and do not depend on the vectors $\check{L}_i$, $i = 1,\ldots,I$. \qed

\subsection{Proof of Lemma~\ref{lem:log}} \label{sec:lem_log}
The cornerstone of this proof is finiteness of the corresponding moments for the busy periods of the queues in isolation, which we check with the help of the auxiliary Lemmas~\ref{lem:Ef(tau)} and~\ref{lem:convex} that follow below together with their proofs.

\begin{lemma} \label{lem:Ef(tau)}
Suppose that a~function $f(\cdot) \colon \nneg \to \nneg$ is bounded in a~finite~interval $[0,T]$ and non-decreasing in~$[T, \infty)$, and that $f(x) \to \infty$ as $x \to \infty$. Suppose also that, for some (and hence for all) $c > 1$,
\begin{equation} \label{eq:log_1}
\limsup_{x \to \infty} f(cx) / f(x) < \infty.
\end{equation}
Consider an i.i.d.~sequence $\{ Y^\ssn \}_{n \in \ntr}$ of non-negative, non-degenerate at zero random variables, and the renewal process 
\[
Y(t) = \max\{n \in \zplus \colon \sum_{k=1}^n Y^\ssk \leq t  \}, \quad t \in \nneg.
\]
Let $\tau$ be a non-negative random variable which may depend on the sequence $\{Y_n\}_{n \in \ntr}$. Assume that  ${\mathbf E} f(\tau )<\infty$. Then ${\mathbf E} f(Y(\tau ))$ is finite too.
\end{lemma}


{\it Proof
.} Without loss of generality, we can assume that the function $f (\cdot)$ is non-decreasing in the entire domain $[0,\infty)$ (otherwise, instead of $f(\cdot)$, one can consider $\wtil{f}(\cdot) = \sup_{0 \leq y \leq \cdot} f(y)$), and also that $f(\cdot)$ is right-continuous.

First we show that, if~\eqref{eq:log_1} holds for some $c > 1$, then it holds for any $c' > 1$. For $c' = c^k$, $k \in \ntr$, we have
\[
\limsup_{x \to \infty} \frac{f(c^k x)}{f(x)} \leq \limsup_{x \to \infty} \frac{f(c^k x)}{f(c^{k-1} x)} \limsup_{x \to \infty} \frac{f(c^{k-1} x)}{f(c^{k-2} x)} \ldots \limsup_{x \to \infty} \frac{f(c x)}{f(x)} < \infty.
\]
Then, for $c' > 1$ other than powers of~$c$, \eqref{eq:log_1} follows by the monotonicity of $f(\cdot)$.

Condition~\eqref{eq:log_1} also implies that
\begin{equation}
\lim_{x \to \infty} \log(f(x)) / x = 0. \label{eq:log_2}
\end{equation}
Indeed, in~\eqref{eq:log_1} take $c = e$, the exponent. Since $M := \limsup_{x \to \infty} f(e x) / f(x) < \infty$, there exists a~large enough $T' > 0$ such that $\sup_{x \in [T',\infty)} f(e x) / f(x) \leq 2M$.
Note that any $x \in [e T', \infty)$ admits a~unique representation $x = e^{k(x)} y(x)$, where $y(x) \in [T', e T')$ and $k(x) \in \ntr$. Hence, for any $x \in [e T', \infty)$,
\[
f(x) = \frac{f(e^{k(x)} y(x))}{f(e^{k(x)-1} y(x))} \frac{f(e^{k(x)-1} y(x))}{f(e^{k(x)-2} y(x))} \ldots \frac{f(e y(x))}{f(y(x))} f(y(x)) \leq (2M)^{k(x)} f(e T')
\]
and
\[
\frac{\log( f(x) )}{x} \leq \frac{k(x) \log (2M) + \log (f(e T'))}{T' e^{k(x)}},
\]
implying~\eqref{eq:log_2}.

Now define the pseudo-inverse function
\[
f^{(-1)}(y) := \inf \{ x \in [0,\infty) \colon f(x) \geq y \}, \quad y \in [0,\infty).
\]

For any $c' > 0$, we have
\begin{align*}
\ex f(Y(\tau)) &\leq \sum_{n \in \zplus} \pr \{ f(Y(\tau)) \geq n \} \\
&\leq \sum_{n \in \zplus} \pr \{ Y(\tau) \geq f^{(-1)}(n) \} \leq \sum_{n \in \zplus} \pr \{ \sum_{k=1}^{\lc f^{(-1)}(n) \rc} Y^\ssk \leq \tau \} \\
&\leq \underbrace{\sum_{n \in \zplus} \pr \{ \sum_{k=1}^{\lc f^{(-1)}(n) \rc} Y^\ssk \leq c' \lc f^{(-1)}(n) \rc \}}_{\displaystyle{=:\Sigma_1(c')}} + 
\underbrace{\sum_{n \in \zplus} \pr \{ c' f^{(-1)}(n) < \tau \}}_{\displaystyle{=:\Sigma_2(c')}}.
\end{align*}

By~condition~\eqref{eq:log_1}, $\ex f(\tau / c') < \infty$ for any $c' > 0$, and hence
\[
\Sigma_2(c') \leq \sum_{n \in \zplus} \pr \{ f(\tau / c') \geq n \} \leq 1 + \ex f(\tau / c') < \infty.
\]

We now pick a~$c'$ such that $\Sigma_1(c') < \infty$, and this will finish the proof. By Markov's inequality, $\pr \{ \sum_{k=1}^n Y^\ssk \leq c' n \} = \pr \{ e^{-\sum_{k=1}^n Y^\ssk} \geq e^{-c' n} \} \leq (e^{c'} \, \ex e^{-Y^{(1)}})^n$. Let $c'$ be small enough so that $\wtil{c} := e^{c'} \, \ex e^{-Y^{(1)}} < 1$. Since $\lc f^{(-1)}(n) \rc = m$ implies $n \leq f(m+1)$, we have
\begin{align*}
\Sigma_1(c') \leq \sum_{m \in \zplus} \pr \{ \sum_{k=1}^{m} Y^\ssk \leq c'm \} f(m+1) \leq \frac{1}{\wtil{c}} \sum_{m \in \ntr} \wtil{c}^{\,m} f (m).
\end{align*}
Take an~$\eps \in ( 0, | \log(\wtil{c}) |)$. By~\eqref{eq:log_2}, there exists a~large enough $N \in \ntr$ such that $f(m) \leq e^{m \eps}$ for $m > N$. Then
\begin{align*}
\Sigma_1(c') \leq \frac{1}{\wtil{c}} \sum_{m = 1}^N \wtil{c}^{\,m} f (m) + \frac{1}{\wtil{c}} \sum_{m = N+1}^\infty (\wtil{c} e^\eps)^{m},
\end{align*}
where $\wtil{c} e^\eps = e^{\eps - | \log(\wtil{c}) |} < 1$ by the choice of $\eps$, and hence $\Sigma_1(c') < \infty$. \qed

\begin{lemma} \label{lem:convex}
Consider a~sequence $\{ Y^\ssn \}_{n \in \ntr}$ of non-negative random variables that are identically distributed \textup{(}but not necessarily independent\textup{)}, and also a~$\zplus$-valued random variable $\eta$ that does not depend on~$\{ Y^\ssn \}_{n \in \ntr}$. If $f(\cdot) \colon \nneg \to \mathbb{R}$ is a~convex function, then 
\[
\ex f(\sum_{k=1}^\eta Y^\ssk) \leq \ex f(\eta Y^{(1)}).
\]
\end{lemma}

{\it Proof.}
By the convexity of $f(\cdot)$, for any $n \in \zplus$,
\[
\ex f(\sum_{k=1}^n Y^\ssk) = \ex f(\sum_{k=1}^n \frac{1}{n} (n Y^\ssk)) \leq \sum_{k=1}^n \frac{1}{n} \ex f(n Y^\ssk) = \ex f(n Y^{(1)}).
\]
Then, by the independence between $\{ Y^\ssn \}_{n \in \ntr}$ and $\eta$,
\begin{align*}
\ex f(\sum_{k=1}^\eta Y^\ssk) &= \sum_{n \in \zplus} \pr \{ \eta = n \} f(\sum_{k=1}^n Y^\ssk)  \\
&\leq  \sum_{n \in \zplus} \pr \{ \eta = n \} \, \ex f(n Y^{(1)}) = \ex f(\eta Y^{(1)}). \tag*{\qed}
\end{align*}

Now we proceed with the proof of Lemma~\ref{lem:log}. It suffices to show that
\[
\ex f(L_{i,j}) < \infty, \quad \text{for all $i$ and $j$},
\]
where
\[
f(x) = \left\{
\begin{array}{ll}
0, & x \in [0,1], \\
x \log{x}, & x \in [1,\infty).
\end{array}
\right.
\]

Note that the function $f(\cdot)$ is convex: in $(1, \infty)$, its derivative $\log(\cdot)+1$ is non-decreasing, and in the other points, it is easy to check the definition of convexity.

Also note that
\begin{equation} \label{eq21}
f(x y) \leq x f(y) + y f(x), \quad x,y \in \nneg.
\end{equation}

The rest of the proof is divided into three parts. The two key steps are to show that the $f$-moments of the visit duration~$V_i$ and the same type visit offspring~$\check{L}_{i,i}$ are finite. Then the finiteness of the $f$-moments of the session offspring $L_{i,j}$ follows easily.

\paragraph{Finiteness of $\ex f(V_i)$}  It suffices to show that, in the $M/G/1$ queue with the arrival process $E_i(\cdot)$ and service times $B_i^\ssn$, $n \in \ntr$, the $f$-moment of the busy period is finite. Suppose that at time $t=0$, there is one customer in the queue, and his/her service time $B_i^\sso$ is equal in distribution to $B_i$ and is independent from $E_i(\cdot)$ and $\{ B_i^\ssn \}_{n \in \ntr}$. Let
\begin{gather*}
\tau_i = \min\{t \in \pos \colon \text{the queue is empty at $t$} \}, \\
\tau_i(0) = 0, \quad \tau_i(1) = B_i^\sso, \quad \tau_i(k+2) = \tau_i(k+1) + \sum_{n = E_i(\tau_i(k))+1}^{E_i(\tau_i(k+1))} B_i^\ssn, \quad k \in \zplus.
\end{gather*}
Whilst $\tau_i$ is a~busy period, $\tau_i(k)$ is equal in distribution to the visit duration in queue~$i$ of the polling system given that the service discipline in that queue is $k$-gated, and 
\[
\tau_i(k) \uparrow \tau_i \quad \text{a.s. as $k \to \infty$}.
\]
Now we show that the moments $\ex f(\tau_i(k))$, $k \in \zplus$, are bounded. Then the finiteness of $\ex f(\tau_i)$ follows by the continuity of $f(\cdot)$ and the dominated convergence theorem.

Mimicking~\eqref{eq20} , we have
\[
\tau_i(k+1) \cop B_i^\sso + \sum_{l=1}^{E_i(B_i^\sso)} \tau_i(k)^\ssl, \quad k \geq 1,
\]
where $\tau_i(k)^\ssl$, $l \in \ntr$, are i.i.d.\ copies of $\tau_i(k)$ that are independent from $B_i^\sso$ and $E_i(\cdot)$. Then, by the monotonicity and convexity of $f(\cdot)$, and the auxiliary Lemma~\ref{lem:convex} combined with~\eqref{eq21}, 
\begin{align*}
\ex f(\tau_i(k)) \leq \ex f(\tau_i(k+1)) &\leq \frac{1}{2} \ex f(2 B_i^\sso) + \frac{1}{2} \ex f(2 \sum_{l=1}^{E_i(B_i^\sso)} \tau_i(k)^\ssl) \\
& \leq \frac{1}{2} \ex f(2 B_i^\sso) + \frac{1}{2} \ex f(2 E_i(B_i^\sso) \tau_i(k)^{(1)}) \\
& \leq \frac{1}{2} \ex f(2 B_i^\sso) +\frac{\lmb_i}{\mu_i} \ex f(\tau_i(k)) + \frac{1}{2} \ex \tau_i(k) \ex f(2 E_i(B_i^\sso)),
\end{align*}
where $\ex f(2 E_i(B_i^\sso)) < \infty$ by the auxiliary Lemma~\ref{lem:Ef(tau)}, and $\ex \tau_i(k) \leq 1/(\mu_i - \lmb_i)$ by~\eqref{eq22}.

Thus, for all $k \geq 2$,
\[
\ex f(\tau_i(k)) \leq c / (1-\lmb_i/\mu_i),
\]
where
\[
c = \ex f(2 B_i^\sso)/2 + \ex f(2 E_i(B_i^\sso)) / (2(\mu_i - \lmb_i)) < \infty.
\]

\paragraph{Finiteness of $\ex f(\check{L}_{i,i})$} Note that $L_{ii}$ is bounded stochastically from above by the number of service completions during the busy period of the $M/G/1$ queue introduced when proving the finiteness of $\ex f(V_i)$. The number of service completions during the first busy period $\tau_i$ is given by $1+E_i(\tau_i)$, and the finiteness of $\ex f(1+E_i(\tau_i))$ follows by the auxiliary Lemma~\ref{lem:Ef(tau)}.

\paragraph{Finiteness of $\ex f(L_{i,j})$} This part of the proof uses mathematical induction. Now that we have shown the finiteness of the moments $\ex f(\check{L}_{i,i})$,~\eqref{eq23} and Lemma~\ref{lem:Ef(tau)} imply that
\begin{equation} \label{eq25}
\ex f(\check{L}_{i,j}) < \infty \quad \text{for all $i$ and $j$}.
\end{equation}
Then we have the basis of induction: $\ex f(L_{I,j}) = \ex f(\check{L}_{I,j}) < \infty$ for all $j$. Suppose that $\ex f(L_{k,j}) < \infty$ for $k = i+1, \ldots, I$ and all $j$. Then the induction step (from $i+1$ to $i$) follows by~\eqref{eq24}, the convexity of $f(\cdot)$, Lemma~\eqref{lem:convex} combined with~\eqref{eq21}, and~\eqref{eq25}.
\qed

\section{Proofs for Section~\ref{sec:fluid_limit}} \label{sec:proof_fluid_limit}
First we make preparations in Sections~\ref{sec:add_notation} and~\ref{sec:preliminary}, and then proceed with the proofs of Theorems~\ref{th:eta_n} and~\ref{th:fluid_limit} in Sections~\ref{sec:proof_th_eta_n} and~\ref{sec:proof_th_fluid_limit}, respectively.

\subsection{Additional notation} \label{sec:add_notation}
In this section we introduce a~number of auxiliary random objects that we operate with when proving the a.s. convergence results of the paper.

\paragraph{Queue length dynamics}
Define the renewal processes
\[
B_i(t) := \max \{ n \in \zplus \text{ such that } \sum_{k=1}^n B_i^\ssk \leq t \}, \quad t \in \nneg,
\]
and the processes
\[
I_i(t) := \int_0^t \ind \{ \text{queue~$i$ is in service at time~$s$} \} \, ds, \quad t \in \nneg,
\]
that keep track of how much time the server has spent in each of the queues.

Then the number of queue~$i$ customers that have departed up to time~$t$ is given by
\[
D_i(t) := B_i(I_i(t)).
\]
Most of the a.s. convergence results of the paper we derive from the basic equations
\begin{equation*} 
Q_i(\cdot) = E_i(\cdot) - D_i(\cdot).
\end{equation*}

The preliminary results of Section~\ref{sec:preliminary} depend on when the system empties for the last time. The number of indices $n$ such that $\vc{Q}(t^\ssn) = 0$ has a geometric distribution with parameter $q_G < 1$ (see Lemma~\ref{lem:supercritical}). Denote by $\nu$ the last such index, i.e. 
\[
\vc{Q}(t^{(\nu)}) = \vo \quad \text{and} \quad \vc{Q}(t^\ssn) \neq \vo \quad  \text{for all $n > \nu$}.
\]

\paragraph{Ancestor-descendant relationships between customers} By the following three rules, we define the binary relation ``{\it is a~descendant of}$\mkern 5mu$" on the set of~customers:
\begin{itemize}
\item each customer is a~descendant of him-/herself;
\item if customer~$2$ arrives while customer~$1$ is receiving service (the two customers are allowed to come from different queues), then customer~$2$ is a~descendant of customer~$1$;
\item if customer~$2$ is a~descendant of customer~$1$, and customer~$3$ a~descendant of customer~$2$, then customer~$3$ is a~descendant of customer~$1$.
\end{itemize}

Now suppose that a~customer is in position~$k$ in queue~$i$ at the beginning of visit~$n$ to queue~$i$. Denote by $V_i^{(n,k)}$  the amount of time during the visit that his/her descendants are in service, and by $\check{L}_{i,j}^{(n,k)}$ the number of his/her descendants in queue~$j$ at the end of the visit. If a~customer is in position~$k$ in queue~$i$ at the beginning of session~$n$, denote by $L_{i,j}^{(n,k)}$ the number of his/her descendants in queue~$j$ at the end of the session. Introduce also the random vectors 
\[
\check{\vc{L}}_i^{(n,k)} := (\check{L}_{i,1}^{(n,k)},\ldots,\check{L}_{i,I}^{(n,k)}) \quad \text{and} \quad 
\vc{L}_i^{(n,k)}  := (L_{i,1}^{(n,k)},\ldots,L_{i,I}^{(n,k)}).
\]

\subsection{Preliminary results} \label{sec:preliminary} In this section, we characterize the asymptotic behavior of the system at the switching instants~$t_i^\ssn$, laying the basis for Theorem~\ref{th:eta_n} that concerns the bigger scale times $t_i^{(\eta_n)}$. 

From the Kesten-Stigum theorem, we derive the following result for the~$t_1^\ssn$'s.

\begin{lemma} \label{lem:Q(t^n)}
There exists a~positive random variable~$\zeta$ such that
\[
\vc{Q}(t^\ssn) / \rho^n \to \zeta v \quad \text{and} \quad \vc{Q}(t_1^\ssn) / \rho^n \to \zeta v \quad \text{a.s.} \quad \text{as $n \to \infty$}.
\]
The distribution of $\zeta$ is given by
\begin{equation} \label{eq28}
\begin{split}
\pr \{ \zeta \geq x \} =& \frac{1}{1 - q_G} \sum_{n \in \zplus} \pr \{ \nu = n \} 
\sum_{\begin{subarray}{l} \vc{k} \in \zplus^I, \\ |\vc{k}| \geq 1 \end{subarray}} G(\vc{k}) \times \\
&\times \sum_{ \begin{subarray}{l} \vc{l} \leq \vc{k}, \\ |\vc{l}| \geq 1 \end{subarray}} \binom{\vc{k}}{\vc{l}}(\vone - \vc{q})^l \vc{q}^{\vc{k}-\vc{l}}
\pr \{ \sum_{i=1}^I \sum_{j=1}^{l_i}\xi_i^{(j)} \geq \rho^{n+1} x \}, \quad x \in \pos,
\end{split}
\end{equation}
where the random variables $\xi_i^{(j)}$ are the same as in Theorem~\ref{th:eta_n}.
\end{lemma}

\begin{proof} 
Since $t_1^\ssn = t^\ssn$ for $n > \nu$, it suffices to find the a.s. limit of $\vc{Q}(t^\ssn) / \rho^n$.

First we find the asymptotics of the auxiliary MTBP $\{ \vc{Z}^\ssn \}_{n \in \ntr}$ (without immigration) under the assumption that $\vc{Z}^\sso$ is distributed according to $\{G(\vc{k})\}_{\vc{k} \in \zplus^I}$ (the immigration distribution for the MTBP $\{ \vc{Q}(t^\ssn) \}_{n \in \ntr}$ ). 

By Proposition~\ref{th:Kesten_Stigum}, if the distribution of $\vc{Z}^\sso$ is $\{G(\vc{k})\}_{\vc{k} \in \zplus^I}$, we have
\[
\vc{Z}^\ssn / \rho^n \to (\underbrace{ \sum_{\vc{k} \in \zplus^I} \ind \{ \vc{Z}^\sso = \vc{k} \} \sum_{i=1}^I \sum_{j=1}^{k_i} \zeta_i^\ssj }_{\displaystyle{=: \zeta_G}}) \, \vc{v} \quad \text{a.s. as $n \to \infty$},
\]
where $\zeta_i^\ssj$, $j \in \ntr$, are i.i.d.\ copies of $\zeta_i$, and the sequences $\{ \zeta_i^\ssj \}_{j \in \ntr}$, $i = 1, \ldots, I$, are mutually independent and also independent from~$\vc{Z}^\sso$.

The distribution of $\zeta_G$ is given by
\begin{align} 
\pr \{ \zeta_G = 0 \} &= q_G,  \nonumber \\
\pr \{ \zeta_G \geq x \} &= \sum_{\begin{subarray}{l} \vc{k} \in \zplus^I, \\ |\vc{k}| \geq 1 \end{subarray}} G(\vc{k}) \sum_{\begin{subarray}{l} \vc{l} \leq \vc{k}, \\ |\vc{l}| \geq 1 \end{subarray}} \binom{\vc{k}}{\vc{l}} p(\vc{l}) \times \label{eq27} \\
& \quad \times \underbrace{\prod_{i=1}^I \pr \{ \zeta_i > 0 \}^{l_i} \pr \{ \zeta_i = 0 \}^{k_i - l_i}}_{\displaystyle{=(\vone - \vc{q})^{\vc{k}-\vc{l}} \vc{q}^\vc{l}}}, \quad x \in \pos, \nonumber
\end{align}
\vspace{-10pt}
where
\begin{align*}
p(\vc{l}) &= \pr \{ \sum_{i=1}^I \sum_{j=1}^{l_i} \zeta_i^\ssj \geq x \ | \ \zeta_i^\ssj > 0 \text{ for all $i$ and $j = 1, \ldots, l_i$} \} \\
&= \pr \{ \sum_{i=1}^I \sum_{j=1}^{l_i} \xi_i^\ssj \geq x \}
\end{align*}
with the random variables $\xi_i^\ssj$ defined in Theorem~\ref{th:eta_n}.

Now, on the event $\{ \nu = N \}$,
\[
\vc{Q}^{(N+1+n)} / \rho^n \to \zeta_N \vc{v} \quad \text{a.s. as $n \to \infty$},
\]
where
\[
\pr \{ \zeta_N \in \cdot \} = \pr \{ \zeta_G \in \cdot | \vc{Z}^\ssn \neq \vo \text{ for all } n \in \zplus \} = \pr \{ \zeta_G \in \cdot | \zeta_G > 0 \}.
\]
Then
\[
\vc{Q}(t^\ssn) / \rho^n \to (\underbrace{ \sum_{N \in \zplus} \ind \{ \nu = N \} \zeta_N / \rho^{N+1} }_{\displaystyle{=: \zeta}}) \, \vc{v} \quad \text{a.s. as $n \to \infty$},
\]
and it is left to check that the distribution of $\zeta$ is given by~\eqref{eq28}.

For $x \in \pos$, we have
\[
\pr \{ \zeta \geq x \} = \sum_{N \in \zplus} \pr\{ \nu = N \} \pr \{ \zeta_G \geq \rho^{N+1} x \} / \pr \{ \zeta_G > 0 \},
\]
and then~\eqref{eq28} follows by~\eqref{eq27}.
\end{proof}

To deal with the other $t_i^\ssn$'s, we combine the previous lemma with LLN's.

\begin{lemma} \label{lem:t_i^n} For $i = 1, \ldots, I$, there exist constants $b_i \in \pos$ and $\vc{a}_i = (a_{i,1}, \ldots, a_{i,I}) \in \nneg^I$ such that
\[
t_i^\ssn / \rho^n \to b_i \zeta \quad \text{and} \quad \vc{Q}(t_i^\ssn / \rho^n) \to \zeta \vc{a}_i \quad \text{a.s. as $n \to \infty$}.
\]
The $b_i$'s and $\vc{a}_i$'s are given by
\begin{equation} \label{eq:b_i}
b_1 = \frac{\sum_{i=1}^I v_i / \mu_i}{\sum_{i=1}^I \lmb_i / \mu_i - 1}, \quad b_{i+1} = b_i + (v_i + \lmb_i (b_i - b_1)) \gamma_i, \quad i = 1, \ldots, I,
\end{equation}
and
\begin{equation} \label{eq:a_i}
\vc{a}_1 = \vc{v}, \quad \vc{a}_{i+1} = \vc{a}_i + (b_{i+1} - b_i) \boldsymbol{\lmb} - (b_{i+1} - b_i) \mu_i \vc{e}_i, \quad i = 1, \ldots, I.
\end{equation}
The $\vc{a}_i$'s also satisfy
\begin{equation} \label{eq:a_i_only}
\vc{a}_1 = \vc{v}, \quad \vc{a}_{i+1} = \vc{a}_i - a_{i,i} \vc{e}_i + a_{i,i} \check{\vc{m}}_i, \quad i = 1, \ldots, I.
\end{equation}
\end{lemma}

\begin{remark} \label{rem:a_ova}
As we compare \eqref{eq:b_i}--\eqref{eq:a_i} with \eqref{eq:ovb_i}--\eqref{eq:ova_i}, it immediately follows that
\[
b_i = \alpha \ovb_i \quad \vc{a}_i = \alpha \ovva_i
\]
\end{remark}

{\it Proof of Lemma~\textup{\ref{lem:t_i^n}}.} First we show that the sequences of $t_i^\ssn / \rho^n$ and $\vc{Q}(t_i^\ssn)$ converge a.s., and that their limits satisfy the relations~\eqref{eq:a_i}. Then we derive equations~\eqref{eq:b_i} and \eqref{eq:a_i_only} relying on an~LLN that, generally speaking, guarantees the $b_i$'s to be in-probability-limits only.

\paragraph{Asymptotics of $t_1^\ssn$} By the definition of~$\nu$, which is a.s. finite, we have, for $n > \nu$, 
\[
t_1^\ssn = t^\ssn = \underbrace{\sum_{l=0}^\nu (t_1^\ssl - t^\ssl)}_{\displaystyle{=: \Sgm}} + \sum_{i=1}^I \sum_{k=1}^{D_i(t^\ssn)} B_i^\ssk.
\]
where $\Sgm$ is a.s. finite.

The last equation with $D_i(t^\ssn) = E_i(t^\ssn) - Q_i(t^\ssn)$ plugged in can be transformed into
\[
t_1^\ssn = t^\ssn = \Sigma + t^\ssn \underbrace{ \sum_{i=1}^I \frac{\sum_{k=1}^{D_i(t^\ssn)} B_i^\ssk}{D_i(t^\ssn)} \frac{E_i(t^\ssn)}{t^\ssn} }_{\displaystyle{=: \Sigma_1^\ssn}} - \rho^n \underbrace{ \sum_{i=1}^I \frac{\sum_{k=1}^{D_i(t^\ssn)} B_i^\ssk}{D_i(t^\ssn)} \frac{Q_i(t^\ssn)}{\rho^n}}_{\displaystyle{=: \Sigma_2^\ssn}},
\]
and then into
\begin{equation} \label{eq1} 
t_1^\ssn / \rho^n = t^\ssn / \rho^n = \frac{\Sigma_2^\ssn - \Sigma / \rho^n}{\Sigma_1^\ssn - 1}.
\end{equation}
By the SLLN and Lemma~\ref{lem:Q(t^n)},
\[
\Sigma_1^\ssn \to \sum_{i=1}^I \lmb_i / \mu_i \quad \text{and} \quad \Sigma_2^\ssn \to (\sum_{i=1}^I v_i / \mu_i) \zeta \quad \text{a.s. as $n \to \infty$},
\]
which, together with \eqref{eq1}, implies that
\begin{equation} \label{eq2}
t^\ssn / \rho^n \to b_1 \zeta \quad \text{and} \quad  t_1^\ssn / \rho^n \to b_1 \zeta \quad \text{a.s. as $n \to \infty$},
\end{equation}
where the value of $b_1$ is the one claimed in the Lemma.

\paragraph{Convergence of $t_i^\ssn / \rho^n$} Note that
\[
t_{i+1}^\ssn - t_i^\ssn = I_i(t^{(n+1)}) - I_i(t^\ssn),
\]
and hence,
\begin{equation} \label{eq3}
\frac{t_{i+1}^\ssn}{\rho^n} = \frac{t_i^\ssn}{\rho^n} + \frac{I_i(t^{(n+1)})}{B_i(I_i(t^{(n+1)}))} \frac{D_i(t^{(n+1)})}{\rho^{n+1}} \rho - \frac{I_i(t^\ssn)}{B_i(I_i(t^\ssn))} \frac{D_i(t^\ssn)}{\rho^n}.
\end{equation}
By the SLLN,
\begin{equation} \label{eq4}
\frac{B_i(I_i(t^\ssn))}{I_i(t^\ssn)} \to \mu_i \quad \text{a.s. as $n \to \infty$} .
\end{equation}
By the SLLN, \eqref{eq1} and Lemma~\ref{lem:Q(t^n)},
\begin{equation} \label{eq5}
\frac{D_i(t^\ssn)}{\rho^n} = \frac{E_i(t^\ssn)}{t^\ssn} \frac{t^\ssn}{\rho^n} - \frac{Q_i(t^\ssn)}{\rho^n} \to (\lmb_i b_1 - v_i) \zeta \quad \text{a.s. as $n \to \infty$}.
\end{equation}
As we put \eqref{eq2}--\eqref{eq5} together, it follows that there exist positive numbers $b_i$ such that
\begin{equation} \label{eq6} 
t_i^\ssn / \rho^n \to b_i \zeta \quad \text{a.s. as $n \to \infty$}, \quad i = 1, \ldots, I+1.
\end{equation}
(The value of $b_1$ is the one claimed in the Lemma, and the equations for the other $b_i$'s that follow from \eqref{eq2}--\eqref{eq5} are not given here since they will not be used anywhere in the proofs.)

\paragraph{Convergence of $\vc{Q}(t_i) / \rho^n$ and \eqref{eq:a_i}} Since, during the time interval $[t_i^\ssn, t_{i+1}^\ssn)$, there are no departures from queues other than~$i$, we have
\begin{equation} \label{eq7}
Q_j(t_{i+1}^\ssn) - Q_j(t_i^\ssn) = E_j(t_{i+1}^\ssn) - E_j(t_i^\ssn) - \ind\{ j = i \} (B_i(I_i(t^{(n+1)})) - B_i(I_i(t^\ssn))).
\end{equation}
By the SLLN and \eqref{eq6},
\begin{equation} \label{eq8}
\frac{E_j(t_{i+1}^\ssn) - E_j(t_i^\ssn)}{\rho^n} \to \lmb_j(b_{i+1} - b_i) \zeta \quad \text{a.s. as $n \to \infty$}.
\end{equation}
By \eqref{eq6},
\begin{equation} \label{eq17}
\frac{I_i(t^\ssn)}{\rho^n} = \frac{\sum_{k=1}^{n-1} (t_{i+1}^\ssk - t_i^\ssk)}{\rho^n} \to \frac{b_{i+1} - b_i}{\rho - 1} \zeta  \quad \text{a.s. as $n \to \infty$},
\end{equation}
which, together with the SLLN, implies that
\begin{equation} \label{eq9}
\frac{B_i(I_i(t^{(n+1)})) - B_i(I_i(t^\ssn))}{\rho^n} \to \mu_i(b_{i+1} - b_i) \zeta \quad \text{a.s. as $n \to \infty$} \quad 
\end{equation}
As we put Lemma~\ref{lem:Q(t^n)} and \eqref{eq7}--\eqref{eq9} together, it follows that 
\begin{equation} \label{eq10}
\vc{Q}(t_i^\ssn) / \rho^n \to \zeta \vc{a}_i \quad \text{a.s. as $n \to \infty$}, \quad i = 1, \ldots, I+1,
\end{equation}
where the vectors $\vc{a}_i = (a_{i,1}, \ldots, a_{i,I})$ are given by~\eqref{eq:a_i}.

\paragraph{Proof of \eqref{eq:b_i} and \eqref{eq:a_i_only}} We derive~\eqref{eq:b_i} from the equations
\begin{align}
t_{i+1}^\ssn &= t_i^\ssn + \sum_{k=1}^{Q_i(t_i^\ssn)} V_i^{(n,k)} \label{eq11}, \\
Q_i(t_i^\ssn) &= Q_i(t_1^\ssn) + E_i(t_i^\ssn) - E_i(t_1^\ssn). \label{eq12}
\end{align}

To~\eqref{eq11}, we apply the following form of the LLN (the proof is postponed to the appendix).

\begin{statement} \label{st:LLN}
Let a~random variable $Y$ have a~finite mean value and, for each $n \in \ntr$, let $Y_n^\ssk$, $k \in \ntr$, be i.i.d.\ copies of $Y$. Let $\tau_n$, $n \in \ntr$, be $\zplus$-valued random variables such that $\tau_n$ is independent of the sequence $\{Y_n^\ssk\}_{k \in \ntr}$ for each~$n$ and $\tau_n \to \infty$ in probability as $n \to \infty$. Finally, let a~sequence $\{T_n\}_{n \in \ntr}$ of positive numbers increase to $\infty$. If there exists an~a.s.~finite random variable $\tau$ such that $\tau_n / T_n \to \tau$ in probability as $n \to \infty$, then
\[
\sum_{k=1}^{\tau_n} Y_n^\ssk / T_n \to \tau \ex Y \quad \text{in probability as $n \to \infty$}.
\]
\end{statement}

By \eqref{eq11} and Statement~\ref{st:LLN},
\begin{equation} \label{eq13}
b_{i+1} - b_i = a_{i,i} \gamma_i.
\end{equation}

By \eqref{eq12} and the SLLN,
\begin{equation} \label{eq14}
a_{i,i} = v_i + \lmb_i(b_i - b_1).
\end{equation}

Then~\eqref{eq:b_i} follows as we plug \eqref{eq14} into \eqref{eq13}.

Finally,~\eqref{eq:a_i_only} follows as we apply Statement~\ref{st:LLN} to the equation
\begin{equation*}
\vc{Q}(t_{i+1}^\ssn) = \vc{Q}(t_i^\ssn) - Q_i(t_i^\ssn) \vc{e}_i + \sum_{k=1}^{Q_i(t_i^\ssn)} \check{\vc{L}}_i^{(n,k)}. \tag*{\qed}
\end{equation*}

\subsection{Proof of Thereom~\ref{th:eta_n}} \label{sec:proof_th_eta_n}
This proof converts the results of Lemma~\ref{lem:t_i^n} using the following tool.

\begin{lemma} \label{lem:eta_n}
Suppose that random variables $Y^\ssn$, $n \in \mathbb{Z}$, and $Y$ are such that
\[
Y^\ssn / \rho^n \to Y \zeta \quad \text{a.s. as $n \to \infty$}.
\]
Then, for all $k \in \mathbb{Z}$,
\[
Y^{(\eta_n + k)} / \rho^n \to Y \rho^{\lf \log_\rho (\alpha \zeta) \rf} \rho^k / \alpha \quad \text{a.s. as $n \to \infty$}.
\]
\end{lemma}

{\it Proof.} First we show that,
\begin{equation} \label{eq15} 
\text{a.s.}, \quad n - \eta_n = \lf \log_\rho (\alpha \zeta) \rf \quad \text{for all $n$ big enough}.
\end{equation}

Indeed, we have $\log_\rho (t^\ssn) - n = \log_\rho (\alpha \zeta) + \dlt^\ssn$, where $\dlt^\ssn \to 0$ a.s. as $n \to \infty$. Then $\eta_n  = \min \{ k \colon \log_\rho (t^\ssk) \geq n \} = \min \{ k \colon k +\dlt^\ssk \geq n - \log_\rho (\alpha \zeta) \}$. Introduce the event $\Omega' := \{ \dlt^\ssn \to 0, \log_\rho (\alpha \zeta) \notin \mathbb{Z} \}$. When estimated at any $\omega \in \Omega'$, $\eta_n = \lc n -  \log_\rho (\alpha \zeta) \rc = n - \lf \log_\rho (\alpha \zeta) \rf$ for all~$n$ big enough; and $\pr \{ \Omega' \} = 1$ since the distribution function of random variable $\zeta$ is continuous in~$\pos$ (see~\eqref{eq28}, where the random variables $\xi_i^\ssj$ have continuous densities on~$\pos$ by Proposition~\ref{th:Kesten_Stigum}).

Now fix a~$k \in \mathbb{Z}$. By~\eqref{eq15}, 
\[
\frac{Y^{(\eta_n + k)}}{\rho^n}  = \frac{Y^{(\eta_n + k)}}{\rho^{\eta_n+k}} \frac{\rho^k}{\rho^{n - \eta_n}} \to Y \zeta \frac{\rho^k}{\rho^{\lf \log_\rho (\alpha \zeta) \rf}} \quad \text{a.s. as $n \to \infty$},
\]
where
\begin{equation*}
\frac{\zeta}{\rho^{\lf \log_\rho (\alpha \zeta) \rf}} = \frac{\rho^{ \log_\rho (\alpha \zeta) }}{\rho^{\lf \log_\rho (\alpha \zeta) \rf}}  \frac{1}{\alpha} = \rho^{ \{ \log_\rho (\alpha \zeta) \} } / \alpha. \tag*{\qed}
\end{equation*}

Now we proceed with the proof of Theorem~\ref{th:eta_n}. 

Lemmas~\ref{lem:t_i^n} and~\ref{lem:eta_n} imply that the convergence~\eqref{eq16} holds with
\[
\xi := \rho^{ \{ \log_\rho (\alpha \zeta) \} }.
\]

By definition, $\xi$ takes values in $[1, \rho)$, and it is left to calculate its distribution.

Fix an $x \in [1,\rho)$. Since
\begin{align*}
\pr \{ \xi \geq x \} = \pr \{ \{ \log_\rho (\alpha \zeta) \} \geq \log_\rho x \} 
&= \sum_{m \in \mathbb{Z}} \pr \{ m + \log_\rho x \leq \log_\rho (\alpha \zeta) < m + 1\} \\
&= \sum_{m \in \mathbb{Z}} \pr \{ \rho^m x / \alpha \leq \zeta < \rho^{m+1} / \alpha \},
\end{align*}
we have, by Lemma~\ref{lem:Q(t^n)},
\begin{align*}
\pr \{ \xi \geq x \} =& \frac{1}{1 - q_G} \sum_{n \in \zplus} \pr \{ \nu = n \} 
\sum_{\begin{subarray}{l} \vc{k} \in \zplus^I, \\ |\vc{k}| \geq 1 \end{subarray}} G(\vc{k}) \sum_{ \begin{subarray}{l} \vc{l} \leq \vc{k}, \\ |\vc{l}| \geq 1 \end{subarray}} \binom{\vc{k}}{\vc{l}}(\vone - \vc{q})^\vc{l} \vc{q}^{\vc{k}-\vc{l}} \times \\
&\times \underbrace{\sum_{m \in \mathbb{Z}} \pr \{ \rho^{m+n+1} x / \alpha \leq \sum_{i=1}^I \sum_{j=1}^{l_i} \xi_i^{(j)} < \rho^{m+n+2} / \alpha \}}_{\displaystyle{=: \Sigma_{n,l}}}.
\end{align*}
Note that $\Sigma_{n,l}$ does not depend on $n$,
\begin{align*}
\Sigma_{n,l} &= \sum_{m \in \mathbb{Z}} \pr \{ \rho^{m} x / \alpha \leq \sum_{i=1}^I \sum_{j=1}^{l_i} \xi_i^{(j)} < \rho^{m+1} / \alpha \} \\
&= \sum_{m \in \mathbb{Z}} \pr \{ m + \log_\rho x  \leq \log_\rho ( \alpha \sum_{i=1}^I \sum_{j=1}^{l_i} \xi_i^{(j)}) < {m+1} \} \\
&= \pr \{ \{ \log_\rho ( \alpha \sum_{i=1}^I \sum_{j=1}^{l_i} w_i^{(j)}) \}  \geq  \log_\rho x\},
\end{align*}
and this finishes the proof. \qed

\subsection{Proof of Theorem~\ref{th:fluid_limit}} \label{sec:proof_th_fluid_limit}
The proof consists of several steps. Throughout the proof, we assume that the function~$\ovvq(\cdot)$ is defined by~\eqref{eq:ovQ_v2}. First we show that the process $\xi \ovvq(\cdot / \xi)$ coincides a.s. with the pointwise limit of the scaled processes $\ovvq^\ssn(\cdot)$. Then we check that~$\ovvq(\cdot)$ satisfies~\eqref{eq:ovQ_v1} and is continuous. Finally, we prove that the pointwise convergence of the  processes $\ovvq^\ssn(\cdot)$ implies their uniform convergence on compact sets.

\paragraph{Pointwise convergence} To start with, we define the auxiliary event $\Ome'$ on which, as $n \to \infty$,
\begin{align*}
t_i^{(\eta_n + k)} / \rho^n \to  \rho^k \ovb_i \xi \quad \text{and} \quad 
\vc{Q}(t_i^{(\eta_n + k)}) / \rho^n \to \xi \rho^k \ova_i, \quad &i = 1, \ldots, I+1, \quad k \in \mathbb{Z}, \\
I(t_i^{(\eta_n + k)}) / \rho^n \to \rho^k \frac{\ovb_{i+1} - \ovb_i}{\rho(\rho - 1)} \xi, \quad &i = 1, \ldots, I, \quad k \in \mathbb{Z}, \\
E_i(t) / t  \to \lmb_i \quad \text{and} \quad B_i(t) / t  \to \mu_i, \quad &i = 1, \ldots, I.
\end{align*}
By theorem~\ref{th:eta_n}, \eqref{eq17} and the SLLN, $\pr \{ \Ome' \} = 1$.

We will now show that on $\Ome'$, as $n \to \infty$,
\begin{equation} \label{eq18}
\ovvq^\ssn(t) \to \xi \ovvq(t / \xi) \quad \text{for all $t \in \nneg$},
\end{equation}
where $\ovvq(\cdot)$ is given by~\eqref{eq:ovQ_v2}.

Fix a~queue number~$i$ and an~outcome $\ome \in \Ome'$. All random objects in the rest of this part of the proof will be evaluated at this~$\ome$.

For $t = 0$, the convergence \eqref{eq18} holds since the system starts empty. For $t > 0$, we consider the three possible cases.

{\it Case~$1 \colon$ $t \in [\rho^k \ovb_i \xi, \rho^k \ovb_{i+1} \xi)$ for a~$k \in \mathbb{Z}$.} By the definition of~$\Ome'$, for all~$n$ big enough,
\[
t_i^{(\eta_n + k)} / \rho^n < t < t_{i+1}^{(\eta_n + k)} / \rho^n,
\]
implying that queue~$i$ is in service during $[t_i^{(\eta_n + k)}, \rho^n t)$, and hence
\[
Q_i(\rho^n t) = Q_i(t_i^{(\eta_n + k)}) + (E_i(\rho^n t) - E_i(t_i^{(\eta_n + k)}) - (D_i(\rho^n t) - D_i(t_i^{(\eta_n + k)}),
\]
where
\[
D_i(\rho^n t) - D_i(t_i^{(\eta_n + k)}) = B_i(I_i(t_i^{(\eta_n + k)})+(\rho^n t - t_i^{(\eta_n + k)})) - B_i(I_i(t_i^{(\eta_n + k)})).
\]
Again by the definition of $\Ome'$, the last two equations imply that, as $n \to \infty$,
\[
\ovq_i^\ssn(t) \to \rho^k  \ova_{i,i} \xi + \lmb_i(t - \rho^k \ovb_i \xi) - \mu_i(t - \rho^k \ovb_i \xi) = \xi \ovq_i(t / \xi).
\]

{\it Case~$2 \colon$ $t \in [\rho^k \ovb_{i+1} \xi, \rho^{k+1} \ovb_i \xi)$ for a~$k \in \mathbb{Z}$.} In this case, for all~$n$ big enough,
\[
t_{i+1}^{(\eta_n + k)} / \rho^n < t < t_i^{(\eta_n + k + 1)} / \rho^n,
\]
and hence, queue~$i$ is not in service during $[\rho^n t, t_i^{(\eta_n + k + 1)})$, i.e.
\[
Q_i(t_i^{(\eta_n + k + 1)}) = Q_i(\rho^n t) + E_i(t_i^{(\eta_n + k + 1)}) - E_i(\rho^n),
\]
implying that
\[
\ovq_i^\ssn(t) \to \rho^{k+1} a_{i,i} \xi - \lmb_i(\rho^{k+1} \ovb_i \xi - t) = \xi \ovq_i(t / \xi).
\]

{\it Case~$3 \colon$ $t = \rho^k \ovb_i \xi$ for a~$k \in \mathbb{Z}$.} Since, as $n \to \infty$,
\[
t_{i+1}^{(\eta_n + k - 1)} / \rho^n \to \rho^{k-1} \ovb_{i+1} \xi, \quad t_i^{(\eta_n + k)} / \rho^n \to t \quad \text{and} \quad t_{i+1}^{(\eta_n + k)} / \rho^n \to \rho^k \ovb_{i+1} \xi
\]
and  the limits satisfy the inequality
\[
\rho^{k-1} \ovb_{i+1} \xi < t< \rho^k \ovb_{i+1} \xi,
\]
all $n$ big enough fall into the two sets
\[
\mathcal{N}_1 := \{ n \colon t_i^{(\eta_n + k)} \leq \rho^n t < t_{i+1}^{(\eta_n + k)}\}
\quad \text{and} \quad 
\mathcal{N}_2 := \{ n \colon t_{i+1}^{(\eta_n + k - 1)} < \rho^n t < t_i^{(\eta_n + k)}\}.
\]

For $l = 1,2$, we have to check that, if the set $\mathcal{N}_l$ is infinite, then
\begin{equation} \label{eq19}
\ovq_i^\ssn(t) \to \rho^k \ova_{i,i} \xi \quad \text{as $n \to \infty$, $n \in \mathcal{N}_l$}.
\end{equation}

For $l=1$,~\eqref{eq19} follows along the lines of Case~$1$. For $l = 2$, we can prove~\eqref{eq19} following the lines of Case~$2$ and replacing $k+1$ with~$k$.

\paragraph{Equivalence of \eqref{eq:ovQ_v1} and \eqref{eq:ovQ_v2}} Let $\wtil{\vc{Q}}(\cdot) = (\wtil{Q}_1, \ldots, \wtil{Q}_I)(\cdot)$ be the unique solution to~\eqref{eq:ovQ_v1}, whereas~$\ovvq(\cdot)$, as before, is given by~\eqref{eq:ovQ_v1}. Fix a~queue number~$i$. The slopes of $\ovq_i(\cdot)$ and $\wtil{Q}_i(\cdot)$ coincide everywhere. Also $\ovq_i(0) = 0 = \wtil{Q}_i(0)$, and $\ovq_i(\rho^k \ovb_i) = \rho^k \ova_{i,i} = \wtil{Q}_i(\rho^k \ovb_i)$, $k \in \mathbb{Z}$. Then it is left to check that
\begin{equation} \label{eq19}
\ovq_i(\rho^k \ovb_j) = \rho^k \ova_{j,i} = \wtil{Q}_i(\rho^k \ovb_j), \quad j \neq i, \quad k \in \mathbb{Z}.
\end{equation}
We have,
\begin{equation*}
\begin{array}{llll}
\rho^k \ovb_j \in [ \rho^{k-1} \ovb_{i+1}, \rho^k \ovb_i) & \text{and} & \ovq_i(\rho^k \ovb_j)  = \rho^k (\ovb_i - \lmb_i(\ovb_i - \ovb_j)), & j < i, \\
\rho^k \ovb_j \in [ \rho^k \ovb_{i+1}, \rho^{k+1} \ovb_i) & \text{and} & \ovq_i(\rho^k \ovb_j)  = \rho^k (\rho \ovb_i - \lmb_i(\rho \ovb_i - \ovb_j)), & j > i.
\end{array}
\end{equation*}

Then \eqref{eq19} follows from the equations
\begin{equation*}
\begin{array}{rll}
Q_i(t_i^\ssn) \!\!\!\!\!&= Q_i(t_j^\ssn) + E_i(t_i^\ssn) - E_i(t_j^\ssn), & j < i, \\
Q_i(t_i^{(n+1)}) \!\!\!\!\!&= Q_i(t_j^\ssn) + E_i(t_i^{(n+1)}) - E_i(t_j^\ssn), & j > i,
\end{array}
\end{equation*}
by Lemma~\ref{lem:t_i^n} and Remark~\ref{rem:a_ova}

\paragraph{Continuity of $\ovq(\cdot)$} Fix a~queue number~$i$. As defined by~\eqref{eq:ovQ_v2}, the function $\ovq_i(\cdot)$ might have discontinuities only at $t = 0$ and $t = \rho^k \ovb_{i+1}$, $k \in \mathbb{Z}$.

Note that $\sup_{t \in [\rho^{k-1} \ovb_i, \rho^k \ovb_i)} \ovq_i(t) = \rho^k \ova_{i,i}$, $k \in \mathbb{Z}$. Then $\sup_{t \in (0, \rho^k \ovb_i)} \ovq_i(t) = \sup_{l \in \mathbb{Z}, l \leq k} \rho^k \ova_{i,i} \to 0$ as $k \to -\infty$, and hence, $\ovq_i(t) \to 0 = \ovq_i(0)$ as $t \to 0$.

At $t = \rho^k \ovb_{i+1}$, $k \in \mathbb{Z}$, the function $\ovq_i(\cdot)$ is right-continuous with the left limit given by
$
\lim_{t \uparrow \rho^k \ovb_{i+1}} \ovq_i(t) = \rho^k (\ova_{i,i} + (\lmb_i - \mu_i)(\ovb_{i+1} - \ovb_i)).
$
By~\eqref{eq:ova_i} and~\eqref{eq19}, we have
$
\lim_{t \uparrow \rho^k \ovb_{i+1}} \ovq_i(t) = Q_i(\rho^k \ovb_{i+1}) = \rho^k a_{i+1,i}.
$

\paragraph{Uniform convergence on compact sets} Define the auxiliary event $\Ome''$ on which, as $n \to \infty$, $\ovvq^\ssn(\cdot) \to \xi \ovvq(\cdot / \xi)$ pointwise, and $E_i(\rho^n \cdot)/ \rho^n \to \lmb_i \, \cdot$ uniformly on compact sets, $i = 1, \ldots, I$. As follows from the first part of the proof and the functional SLLN, $\pr\{ \Ome'' \} = 1$. For the rest of the proof, we estimate random objects at an~outcome $\ome \in \Ome''$. Consider the scaled departure processes $D_i(\rho^n \cdot)/ \rho^n = E_i(\rho^n \cdot) / \rho^n - \ovq_i^\ssn(\cdot)$. These processes are monotone and, by the definition of~$\Ome''$, converge pointwise to the continuous functions $\lmb_i \cdot - \xi \ovq_i(\cdot / \xi)$. Then they converge uniformly on compact sets, and the same is true for the processes~$\ovq_i^\ssn(\cdot)$. 
\qed

\section*{Appendix}
\appendix

{\it Proof of Lemma~\ref{lem:supercritical}.} Suppose $\rho \leq 1$. Then, by~\cite[Theorem~7.1]{Harris}, we have $q_i = 1$ for all~$i$ and $q_G = 1$. The latter implies that the queue length process $\vc{Q}(\cdot)$ hits~$\vc{0}$ infinitely many times, and the same holds the workload process. Let $\{ t^{(n_k)} \}_{k \in \zplus}$ be the sequence of consecutive time instants such that $\vc{Q}(t^{(n_k)}) = \vc{0}$. For different $k$, the differences $t^{(n_{k+1})} - t^{(n_k)}$ are bounded from below by the waiting times until the first arrival into the empty system, which are i.i.d.\ random variables distributed exponentially with parameter $\sum_{i=1}^I \lmb_i$. Therefore, $t^{(n_k)} \to \infty$ a.s. as $k \to \infty$. This leads to a~contradiction with the fact that the system is overloaded and its total workload grows infinitely large with time (by the SLLN, $(\sum_{i=1}^I \sum_{k=1}^{E_i(t)} B_i^\ssk - t) / t \to \sum_{i=1}^I \lmb_i / \mu_i - 1 > 0$ a.s. as~$t \to \infty$). Hence, $\rho > 1$, and then \cite[Theorem~7.1]{Harris} implies that $q_i < 1$ for all~$i$ and $q_G < 1$. \qed

{\it Proof of Statement~\ref{st:LLN}.} First we show that
\begin{equation} \label{eq26}
\sum_{k=1}^{\tau_n} Y_n^\ssk / \tau_n \to \ex Y \quad \text{in probability as $n \to \infty$.}
\end{equation}
By the independence between $\tau_n$ and $\{Y_n^\ssk\}_{k \in \ntr}$, for all $N \in \zplus$,
\[
\pr \{ | \sum_{k=1}^{\tau_n} Y_n^\ssk / \tau_n - \ex Y | \geq \eps, \tau_n = N \} = \pr \{ | \sum_{k=1}^{N} Y_1^\ssk / N - \ex Y | \geq \eps \} \pr \{ \tau_n = N \}.
\]
Then, for any $M \in \zplus$,
\[
\pr \{ | \sum_{k=1}^{\tau_n} Y_n^\ssk / \tau_n - \ex Y | \geq \eps \} \leq \pr \{ \tau_n \leq M \} + \sup_{N > M} \pr \{ | \sum_{k=1}^{N} Y_1^\ssk / N - \ex Y | \geq \eps \},
\]
and~\eqref{eq26} follows as we first let $n \to \infty$, and then $M \to \infty$.

Now that we have shown~\eqref{eq26}, the Statement follows by
\begin{align*}
&\pr \{ | \sum_{k=1}^{\tau_n} Y_n^\ssk / T_n - \tau \ex Y | \geq \eps \} \\ 
\leq \  &\pr \{ | \sum_{k=1}^{\tau_n} Y_n^\ssk / \tau_n| |\tau_n / T_n - \tau | \geq \eps/2 \} + \pr \{ \tau | \sum_{k=1}^{\tau_n} Y_n^\ssk / \tau_n - \ex Y | \geq \eps / 2 \} \\
\leq \ &\pr \{ |\tau_n / T_n - \tau | \geq \eps/2 \} + \pr \{ | \sum_{k=1}^{\tau_n} Y_n^\ssk / \tau_n| > C_1 \} \\
&+\pr \{ C_2 | \sum_{k=1}^{\tau_n} Y_n^\ssk / \tau_n - \ex Y | \} + \pr \{ \tau > C_2 \}
\end{align*}
as we first let $n \to \infty$, and then $C_1 \to \infty$, $C_2 \to \infty$.
\qed


\end{document}